\documentclass[oneside, 11pt]{amsart}

\usepackage{amsmath, amssymb, amsthm, hyperref, graphicx,   verbatim, enumerate, xcolor, tikzsymbols, stmaryrd, epsfig, xypic}
\usepackage{array, colortbl}
\usepackage{mathrsfs}  
\usepackage{mathabx}
\usepackage{times}
 
 \DeclareFontFamily{U}{mathc}{}
\DeclareFontShape{U}{mathc}{m}{it}%
{<->s*[1.03] mathc10}{}

\DeclareMathAlphabet{\mathscr}{U}{mathc}{m}{it}

\usepackage[paper=a4paper, marginpar=2.4cm]{geometry}
\usepackage[all]{xy}
\usepackage[utf8]{inputenc}

\usepackage[english]{babel}
 
\newcommand{\e}{\varepsilon}

\newcommand{\set}[1]{\left\{#1\right\}}
\newcommand{\norm}[1]{{\left\Vert#1\right\Vert}}

\newcommand{\rest}[1]{ \arrowvert_{#1}}

\newcommand{\lrpar}[1]{\left(#1\right)}

\newcommand{\inv}{^{-1}}

\newcommand{\C}{\mathbf{C}}
\newcommand{\R}{\mathbf{R}}
\newcommand{\Q}{\mathbf{Q}}
\newcommand{\Z}{\mathbf{Z}}

\newcommand{\bfK}{{\mathbf{K}}}

\renewcommand{\P}{\mathbb{P}}
\newcommand{\A}{\mathbb{A}}

\newcommand{\Aut}{\mathsf{Aut}}

 \newcommand{\Bir}{{\mathsf{Bir}}}

\newcommand{\GL}{{\sf{GL}}}

\theoremstyle{plain}

\newtheorem*{lem}{Lemma}

 \newtheorem*{thmA}{Theorem A}
 \newtheorem*{thmB}{Theorem B}

\theoremstyle{definition}
\newtheorem*{eg}{Example}

\numberwithin{equation}{section}       

\addtocounter{section}{0}             
\numberwithin{equation}{section}       

\title[Conjugate polynomial automorphisms of $\mathbb{C}^2$]{Holomorphically conjugate polynomial automorphisms of $\mathbb{C}^2$ are polynomially conjugate}

\date{\today}

\author{Serge Cantat}
\address{Serge Cantat, IRMAR, Campus de Beaulieu,
b\^atiments 22-23
263 avenue du G\'en\'eral Leclerc, CS 74205
35042  RENNES C\'edex, France}
\email{serge.cantat@univ-rennes1.fr}
 \author{Romain Dujardin}
\address{Romain Dujardin,  Sorbonne Universit\'e, CNRS, Laboratoire de Probabilit\'es, Statistique  et Mod\'elisation  (LPSM), F-75005 Paris, France}
\email{romain.dujardin@sorbonne-universite.fr}
\thanks{
{\small{The research activities of the authors  are partially funded by the European Research Council (ERC GOAT 101053021). The authors benefited from the support of the French government "Investissements d'Avenir" program integrated to France 2030 (ANR-11-LABX-0020-01).}}
}

\begin{document}

\setlength{\parskip}{.2em}
\setlength{\baselineskip}{1.26em}   

\begin{abstract}
We confirm a conjecture of Friedland and Milnor: if two polynomial automorphisms $f$ and $g\in \Aut(\A^2_\C)$ with dynamical degree $>1$ 
are  conjugate by some holomorphic diffeomorphism $\varphi\colon \C^2\to \C^2$, then $\varphi$ is a polynomial automorphism; thus, $f$ and $g$ are conjugate inside $\Aut(\A^2_\C)$. We also discuss a number of variations on this result. 
\end{abstract}

\maketitle

\section{Introduction}

\subsection{Conjugacy classes in $\Aut(\A^2_\bfK)$} In their very influential paper~\cite{friedland-milnor}, Friedland and Milnor study the dynamical and group-theoretic properties of polynomial automorphisms of 
the affine plane defined over the field $\bfK=\R$ or $\C$. In particular, they provide a dichotomy for conjugacy classes in $\Aut(\A^2_\C)$, which we describe below.

For any field $\bfK$, Jung's theorem says that the group $\Aut(\A^2_\bfK)$ is the amalgamated product of the group of {\emph{affine automorphisms}} $\GL_2(\bfK)\ltimes \bfK^2$ and the group of \emph{elementary automorphisms}
$$(x,y)\mapsto (\alpha x+ p(y) , \beta y +\gamma),$$
with amalgamation along their intersection. 
If $f$ is an element of $\Aut(\A^2_\bfK)$, we denote
 by $\deg(f)$ the degree of the formulas defining $f$, and   set 
 $$ \lambda_1(f)=\lim_{n\to +\infty} \deg(f^n)^{1/n}.$$
Friedland and Milnor prove the following:
 \begin{itemize}
\item either $\lambda_1(f)=1$, 
and then $f$ is conjugate to an elementary or an affine automorphism; moreover, if $\bfK$ is algebraically closed, every affine automorphism is conjugate to an elementary automorphism by a linear change of variable; 
\item or $\lambda_1(f)$ is an integer larger than $1$, 
$f$ is conjugate to a \emph{generalized Hénon map}, that is a finite composition of Hénon mappings  $(x,y)\mapsto (ay+p_i(x), x)$ with $\deg(p_i)\geq 2$. In that case, $\lambda_1(f)$ is the product of the degrees of the $p_i$.
\end{itemize} 

As explained in~\cite{friedland-milnor}, it follows that 
the dynamics of $f$ is interesting only when $f$ 
is conjugate to a generalized Hénon map; for instance, for $f\in \Aut(\A^2_\C)$, the topological entropy of the map $f\colon  \C^2\to \C^2$ is equal to $\log(\lambda_1(f))$ (see~\cite{smillie:entropy_henon}).

The Bass-Serre theory shows that $\Aut(\A^2_\bfK)$ acts on a simplicial tree, with stabilizers of vertices corresponding to conjugates of the affine group and the elementary group. Then, the classification of Friedland and Milnor corresponds to the following dichotomy: $\lambda_1(f)=1$ if and only if 
the induced action of $f$ on the Bass-Serre tree fixes some vertex, 
while $\lambda(f)>1$ when $f$  acts 
as a non-trivial translation on some invariant geodesic axis. In accordance with this description, we will refer to these two cases as \emph{elliptic} and \emph{loxodromic}, respectively. 
   
\subsection{Conjugation under biholomorphisms and rigidity} Let us now suppose that $\bfK=\C$. Friedland and Milnor  also study the following question: {\emph{if two polynomial automorphisms are conjugate by some holomorphic diffeomorphism of $\C^2$, are they also conjugate by an element of $\Aut(\A^2_\C)$}}? Automorphisms with distinct dynamical degrees cannot be conjugate by a homeomorphism $\C^2 \to \C^2$, because their topological entropies are distinct.
In particular, elliptic and loxodromic classes should be treated separately. 
  A complete answer is given in~\cite{friedland-milnor} for elementary automorphisms: the answer is {\emph{yes}} if such an automorphism admits a periodic point, and  {\emph{no}}  otherwise~\cite[Theorem 6.10 and Lemma 6.12]{friedland-milnor}. For generalized Hénon maps, Friedland and Milnor also prove that the answer is yes for maps of degree 2 and 3, and conjecture that the same result holds in the general case~\cite[Theorem 7.1]{friedland-milnor}. In this note we confirm this conjecture. 
  
\begin{thmA}
Let $f$ and $g$ be loxodromic polynomial automorphisms of the affine plane, defined over some subfield 
$\bfK$ of $\C$. 
If $\varphi\colon \C^2\to \C^2$ is a biholomorphism that conjugates $f$ to $g$, i.e.\ $\varphi\circ f\circ \varphi{\inv}=g$, then 
 $\varphi$ is a polynomial automorphism, and it is   defined over a finite extension of $\bfK$. 
 Moreover, one can find  $\psi\in \Aut(\A^2_{\overline{\bfK}})$ such that $\psi\circ f=g\circ \psi$ and  $\deg(\psi)\leq 2^{57}(\deg(f)\deg(g))^{29}$.
\end{thmA}

A natural approach to this problem is to use eigenvalues at periodic orbits and to show that loxodromic automorphisms with the same eigenvalues are conjugate  in $\Aut(\A^2_\C)$.
This is precisely what is done in~\cite{friedland-milnor} in degree $\leq 3$. 
However, for arbitrary degree, such a
multiplier rigidity result  is  not available, so we use a different method. 
In a nutshell, the main ingredient will be the existence and rigidity of 
the canonical  invariant currents of $f$ and $g$; these currents $T^\pm_f$ and $T^\pm_g$ must be 
respected by the conjugacy, i.e.\ $\varphi_\varstar T^\pm_f   = c T^\pm_g$ for some $c>0$, and their rigidity properties will be the key to our proof.

 
When $f=g$, the corresponding result was proven in \cite[Proposition 8.1]{Cantat:Annals} and~\cite{blanc-cantat}; we shall use this particular case below. Related results were also obtained by Bera, Pal and Verma in~\cite{bera-pal-verma}, 
with an approach similar to ours.

\subsection*{Acknowledgement} We thank Christophe Dupont and Sébastien Gouëzel for interesting discussions related to this paper

\section{Proof of Theorem A} 
\subsection{Preliminaries on Hénon maps (see~\cite{friedland-milnor} and~\cite{Sibony:Panorama})} 
Let $f\in \Aut(\A^2_\C)$ be a  generalized Hénon map of degree $d$ (so $d=\lambda_1(f)$). Its birational extension to the projective plane $\P^2(\C)$ 
has an indeterminacy point at $I^+_f:=[0:1:0]$, and it contracts the line at infinity $\set{z=0}$ onto the indeterminacy point $I^-_f:=[1:0:0]$ of $f{\inv}$. 
Let $K^+_f$ (resp.\ $K^-_f$) be the subset of $\C^2$ of points with a bounded forward (resp.\ backward) orbit; its closure in $\P^2(\C)$ intersects the line at infinity at $I^+_f$ (resp.\ at $I^-_f$). Let $\norm{\cdot}$ denote the usual euclidean norm on $\C^2$. Then, the function $G^+_f\colon \C^2\to \R$ defined by
$$ G^+_f(p)=\lim_{n\to +\infty} \frac{1}{d^n}\log^+\norm{f^n(p)}$$
is continuous, plurisubharmonic, and non-negative; it vanishes exactly along $K^+_f$ and  satisfies $ G^+_f\circ f=d G^+_f$. The difference $p\mapsto G^+_f(p)-\log^+\norm{p}$   
extends to a continuous function on $\P^2(\C)\setminus \big\{I^+_f\big\}$, as follows for instance from \cite[Thm 1.7.1]{Sibony:Panorama}.  
 The closed, positive current 
$$
T^+_f=\frac{\sf{i}}{\pi}\partial\overline{\partial}G^+_f
$$
satisfies $f^\varstar T^+_f=dT^+_f$; it is supported on the boundary of $K^+_f$ and up to a positive multiplicative constant, {\emph{it is the unique closed positive current supported by $K^+_f$}} (see~\cite[Theorem 7.12]{fornaess-sibony:1993}). Changing $f$ into $f^{\inv}$, one constructs similar objects  $K^-_f$, $G^-_f$ and $T^-_f$. 

The intersection $K_f=K^+_f\cap K^-_f$ is a compact $f$-invariant subset of $\C^2$, it is maximal for this property, and the topological entropy of $f$ on $K_f$ is equal to $\log(d)$. The automorphism $f$ has infinitely many saddle periodic points, all of them contained in $K_f$ (see~\cite{bls}).
If $q\in \C^2$ is such a saddle periodic point, its stable manifold $W^s(q)$ is the image of a holomorphic, injective, immersion $\C\to \C^2$, {\emph{the image of which is dense in $\partial K^+_f$}}; the current $T^+_f$ can be recovered asymptotically as a current of integration on $W^s(p)$ (see~\cite{bs2}) 
The function $G_f:= \max(G^+_f, G^-_f)$ is continuous and plurisubharmonic and vanishes exactly on $K_f$. Moreover
$$G^+_f(p)\leq \log^+\norm{p}  + O(1)\;  \text{ and } \;  G_f(p) = \log^+\norm{p} + O(1)$$ 
as $\norm{p}$ goes to $+\infty$; see Corollary~{2.6} and Proposition~{3.8} in~\cite{bs1}. 
The product $T^+_f\wedge T^-_f$ is an invariant probability measure, supported in $K_f$; it coincides with 
$(\frac{\sf{i}}{\pi}\partial {\overline{\partial}} G_f)^2$.

\subsection{The proof} 
Conjugating $f$ and $g$ by elements of $\Aut(\A^2_\C)$, we may assume that both of them are generalized Hénon maps. Let
$\varphi\colon \C^2\to \C^2$ be a biholomorphism such that $\varphi \circ f\circ  \varphi\inv = g$. Then $\varphi \circ f^n = g^n\circ  \varphi$ for all $n\in \Z$, $\varphi(K^\pm_f)=K^\pm_g$,
$\varphi(K_f)=K_g$, and
$f$ and $g$ have the same topological entropy, hence the same degree $d$. 

\begin{lem}
There exists a positive real number $c$ such that $\varphi^\varstar T_g^+ = c T_f^+$ and 
$  G_g^+\circ \varphi = c G_f^+$. 
\end{lem}

\begin{proof}
The first statement follows from the fact, recalled above,  that the only positive closed $(1,1)$-currents supported on $K_g^+$ are of the form $cT_g^+$, $c>0$. 

To prove the second fact, we observe that $H = cG_f^+  - G_g^+\circ \varphi$ is a pluriharmonic function 
on $\C^2$
with $K_f^+\subset \set{H=0}$. Assume that $H\not\equiv 0$ and fix a    holomorphic function $h$ 
on $\C^2$ such that $H$ is the real part  $\Re(h)$. 
Then $\set{H=0}$ admits a unique singular foliation by Riemann surfaces: its leaves are the level sets $\set{h  = i \alpha}$, $\alpha\in \R$. On the other hand, for any saddle periodic point $q$ of $f$, the stable manifold $W^s(q)\subset \C^2$ is a connected, 
 immersed, Riemann surface contained 
in $K_f^+$; thus $W^s(q)$ is an irreducible component of $\set{h  = i \alpha}$  for some $\alpha$. This is a contradiction since $W^s(q)$ is dense in the boundary of $K^+_f$, so it is not embedded at any point. 
\end{proof}

Similarly, there exists a constant $c'>0$ such that $G^-_g\circ \varphi = c' G_f^-$. In fact, from 
$$\int \varphi^\varstar (T_g^+\wedge T_g^-) =  \int T_f^+\wedge T_f^- =1$$ we deduce  that  $c' = 1/c$, but this will not be used in what follows. 
From this we infer that 
$$0\leq G_g \circ \varphi  = \max (G_g^+\circ \varphi, G_g^- \circ \varphi) = \max \lrpar{c G_f^+ , c' G_f^-} \leq 
\max(c, c') G_f.
$$  
On the other hand, $G_f (p)  = \log\norm{p}+ O(1)$, and  $G_g(p)= \log\norm{p}+ O(1)$ as $\norm{p}\to \infty$. Since  $\varphi$ is proper, we deduce that
$$\log\norm{\varphi(p)} + O(1)  = G _g(\varphi(p))\leq \max(c, c')  G_f(p) \leq \max(c, c') \log \norm{p} +O(1), $$
that is $\norm{\varphi(p)}\leq C \norm{p}^{e}$ for some constant $C>0$ and for $e=\max(c, c') $. This implies that $\varphi$ is defined by polynomial formulas. Similarly, $\varphi\inv$  is also defined by polynomial formulas, and $\varphi$ is a polynomial automorphism. 

If $\varphi$ and $\psi$ are two automorphisms conjugating $f$ to $g$, then $\varphi\circ\psi\inv$ commutes with $f$. 
In~\cite{Lamy:Tits, blanc-cantat}, it is proven that the centralizer of $f$ in $\Aut(\A^2_\C)$ contains $f^\Z$ as a finite index subgroup.
Thus, given any integer $D\geq 1$, the set $\set{\psi\in \Aut(\A^2_\C) \, ; \; \psi\circ f\circ \psi\inv=g, \; \deg(\psi)\leq D}$ is finite, and it is non-empty for $D\geq \deg(\varphi)$.  
Write  $\psi=(\psi_1, \psi_2)$, with $\psi_1$ and $\psi_2$ in $\C[x,y]$ of degree $\leq D$.  
The constraints  $ \psi\circ f\circ \psi\inv=g$  and $\psi\in \Aut(\A^2_\C)$ correspond to polynomial equations in the coefficients of $\psi_1$ and $\psi_2$; the coefficients of these equations are in the field of definition $\bfK$ of $f$ and $g$. Thus, the solutions lie in a finite extension of $\bfK$.
 
It remains to show that one can find an automorphism $\psi$ that conjugates $f$ to $g$ and has degree at most $2^{57}(\deg(f)\deg(g))^{29}$. From~\cite[Theorem 4.10]{blanc-cantat}, one can find such a conjugacy in the group $\Bir(\P^2_{\overline{\bfK}})$. 
Thus, to conclude one only needs to apply the following lemma of independent interest.
\qed

\begin{lem}
Let $f$ and $g\in \Aut(\A^2_\C)$ be two Hénon automorphisms.
If $\psi$ is a birational transformation of the affine plane that $f$ to $g$, then $\psi$ is an automorphism. 
\end{lem}

\begin{proof} Let us consider $f$, $g$, and $\psi$ as birational transformations of the projective plane $\P^2_\C$.
  Let $E$ be the union of the irreducible curves $E_i\subset  \A^2_\C$ which are contracted by $\psi$
(there strict transform is a point of $\P^2_\C$). From  $\psi\circ f=g\circ \psi$ we deduce that $E$ is $f$-invariant, and since a Hénon map does not preserve any algebraic curve, we deduve that $E$ is empty. Similarly, $\psi^{-1}$ does not contract any curve $F\subset \A^2_\C$. 

It remains to show that $\psi$ does not have any indeterminacy point in $\A^2(\C)$. To see this, we resolve the indeterminacies of $\psi$ by a finite sequence of blow-ups: this provides a smooth projective surface $X$ together with two birational morphisms $\e, \, \eta \colon X\to \P^2_\C$ such that $\psi=\eta\circ \e^{-1}$; we assume that $X$ is minimal for this property.
Suppose that $q\in \A^2(\C)$ is an indeterminacy point of $\psi$ and denote by $D\subset X$ the tree of rational curves which is mapped to $q$ by $\e$. The curve $D$ must intersect an irreducible curve $C\subset X$ which is contracted by $\eta$ and is not contained in $D$; indeed, otherwise, $X$ would not be a minimal resolution of the indeterminacies. But then, $C$ is the strict transform by $\e$ of an irreducible plane curve $C'$ that contains $q$. This curve $C'$ is contracted by $\psi$ and contains $q$, in contradiction with $E=\emptyset$.
\end{proof}
    
 \begin{eg} 
 Fix an integer $m\geq 2$, and a positive integer $D$  which is not the $m$-th power of an integer. 
 One easily sees that the Hénon maps $f(x,y)=(y,x+y^{m+1})$ and $g(x,y)=(y,x+D y^{m+1})$ are conjugate over $\Q(D^{1/m})$ but not over $\Q$. Indeed, $h(x,y)=(\alpha x,\alpha y)$ conjugates $f$ to $g$ if $\alpha^{m}=1/D$, and from~\cite{Lamy:Tits} we know that the centralizer of $f$ is the semi-direct product of $f^\Z$ and the cyclic group of order $m$ generated by $(x,y)\mapsto (e^{2 i \pi/m} x, e^{2i\pi/m} y)$.
 \end{eg}

\section{Discussion and complements}

\newcounter{compteur}
\renewcommand{\thecompteur}{\alph{compteur}}
\newenvironment*{listeitem}
  {\refstepcounter{compteur} 
   \par
     {(\thecompteur) --} }
  {\medskip}

\begin{listeitem}
Theorem~A is still true when $\varphi$ is a proper holomorphic semiconjugacy:

\begin{thmB}
If $f$ and $g$ are loxodromic automorphisms of $\C^2$ and 
$\varphi:\C^2\to \C^2$ is a proper holomorphic map such that $\varphi \circ f = g\circ \varphi$, then 
$\varphi$ is a polynomial  automorphism. 
\end{thmB}

\begin{proof} Note that we also have $\varphi \circ f\inv = g\inv\circ \varphi$. The properness of $\varphi$ implies successively that 
\begin{enumerate}[(1)]
\item $\varphi(\C^2)=\C^2$ by Remmert's proper mapping theorem;
\item  $\varphi\inv(K_g^\pm) = K_f^\pm$; 
\item $\varphi^\varstar (T_g^+)$ (resp. $\varphi^\varstar (T_g^-)$)  
is a positive closed current supported on $K_f^+$ (resp. $K_f^-$).
\end{enumerate}
Indeed, $\varphi^\varstar (T_g^+)$ (resp. $\varphi^\varstar (T_g^-)$)  is non-zero by the first item.  Hence 
$\varphi^\varstar T_g^\pm = c^\pm  T_f^\pm$, with $c^\pm >0$, and exactly as in Theorem A we conclude that $\varphi$ is polynomial. Finally, the Jacobian of $\varphi$ does not vanish, 
because otherwise, by the relation $\varphi \circ f = g\circ \varphi$,
$\set{\mathrm{Jac}(\varphi) = 0}$ would be an $f$-invariant algebraic curve, and such a curve does 
not exist by~\cite[Proposition 4.2]{bs1}. Thus, 
since $\varphi$ is proper, it is a covering, hence an automorphism, and we are done. 
\end{proof}
\end{listeitem}

\begin{listeitem}
Take $\bfK=\R$. If $f\in \Aut(\A^2_\R)$  has no fixed point in $\R^2$ and $f$ preserves the orientation of $\R^2$, then $f$ is   conjugate to a translation 
by some real analytic diffeomorphism of $\R^2$ (see~\cite{cantat-lamy}). Thus, \emph{Theorem~A
fails 
for polynomial automorphisms of $\R^2$.} On the other hand one may expect some rigidity properties 
when $h_{\mathrm{top}}(f\rest{\R^2})>0$, for instance when $K_f\subset \R^2$. 
\end{listeitem}

\begin{listeitem}  If $U$ is a small ball in the complement of $K^+_f\cup K^-_f$, then its orbit under the action of $f$ is wandering. So, we can find a $C^\infty$ diffeomorphism of $\C^2$ that commutes to $f$, is the identity on the complement of the $f$-orbit of $U$, and is not the identity on $U$. This shows that, as it is stated,
\emph{Theorem~A fails in the $C^\infty$ setting.} 

On the other hand, the following lemma shows that the canonical currents are preserved under $C^1$ conjugacy. Thus, the rigidity of positive closed currents with support in $K^+$ is not really needed in the proof of Theorem~A, we could replace it by this lemma. 

\begin{lem}
Let $f$ and $g$ be Hénon automorphisms of the plane $\A^2_\C$.
If $\varphi$ is a $C^1$ diffeomorphism of $\C^2$ conjugating $f$ to $g$, then $\varphi^\varstar T^+_g = cT^+_f$ for some real number $c\neq 0$.
\end{lem}

\begin{proof}[Sketch of proof] For  both $f$ and $g$, 
$T^+$ is a laminar current whose laminar structure 
is entirely defined in terms of the smooth dynamics~\cite{bls}. 
More precisely,   the structure theory of 
strongly laminar currents~\cite[\S 5]{structure} shows that 
 these currents can be viewed in a precise way 
as foliation cycles on a weak lamination, 
 whose leaves are Pesin stable manifolds  and whose transverse measure is induced
 by the unstable conditionals of the unique measure of maximal entropy. It follows that if  $\varphi$ is a $C^1$ diffeomorphism conjugating $f$ and $g$, 
 $\varphi^\varstar T^+_g$ is a closed laminar current subordinated to the same weak lamination as $T^+_f$; more precisely,  $\varphi^\varstar T^+_g$  is
 \begin{itemize}
 \item a closed current;
 \item a laminar current with respect a transverse signed measures (the sign takes care of the orientation of the local leaves, since $\varphi$ may reverse this orientation); 
 \item absolutely continuous with respect to $T_f^+$; in particular it is supported by $K^+_f$.
 \end{itemize}
On the other hand, the  fact that 
 $T_f^+$ is an extremal point in the cone of positive 
 closed currents (see e.g.~\cite[\S 2.2]{Sibony:Panorama})  
  entails that  the transverse invariant measure of $T^+_f$ is ergodic (see~\cite[Cor. 5.8]{structure})
  and this implies that 
$\varphi^\varstar T^+_g$ is a constant multiple of $T^+_f$.  
\end{proof}
\end{listeitem}

 \bibliographystyle{plain}
\bibliography{bib-conjugate}

\begin{thebibliography}{10}

\bibitem{bls}
Eric Bedford, Mikhail Lyubich, and John Smillie.
\newblock Polynomial diffeomorphisms of {${\bf C}^2$}. {IV}. {T}he measure of
  maximal entropy and laminar currents.
\newblock {\em Invent. Math.}, 112(1):77--125, 1993.

\bibitem{bs1}
Eric Bedford and John Smillie.
\newblock Polynomial diffeomorphisms of {${\bf C}^2$}: currents, equilibrium
  measure and hyperbolicity.
\newblock {\em Invent. Math.}, 103(1):69--99, 1991.

\bibitem{bs2}
Eric Bedford and John Smillie.
\newblock Polynomial diffeomorphisms of {${\bf C}^2$}. {II}. {S}table manifolds
  and recurrence.
\newblock {\em J. Amer. Math. Soc.}, 4(4):657--679, 1991.

\bibitem{bera-pal-verma}
Sayani Bera, Ratna Pal, and Kaushal Verma.
\newblock A rigidity theorem for {H}\'{e}non maps.
\newblock {\em Eur. J. Math.}, 6(2):508--532, 2020.

\bibitem{blanc-cantat}
J\'{e}r\'{e}my Blanc and Serge Cantat.
\newblock Dynamical degrees of birational transformations of projective
  surfaces.
\newblock {\em J. Amer. Math. Soc.}, 29(2):415--471, 2016.

\bibitem{Cantat:Annals}
Serge Cantat.
\newblock Sur les groupes de transformations birationnelles des surfaces.
\newblock {\em Ann. of Math. (2)}, 174(1):299--340, 2011.

\bibitem{cantat-lamy}
Serge Cantat and St\'{e}phane Lamy.
\newblock Groupes d'automorphismes polynomiaux du plan.
\newblock {\em Geom. Dedicata}, 123:201--221, 2006.

\bibitem{structure}
Romain Dujardin.
\newblock Structure properties of laminar currents on {$\Bbb P^2$}.
\newblock {\em J. Geom. Anal.}, 15(1):25--47, 2005.

\bibitem{fornaess-sibony:1993}
John~Erik Forn{\ae}ss and Nessim Sibony.
\newblock Complex dynamics in higher dimensions.
\newblock In {\em Complex potential theory ({M}ontreal, {PQ}, 1993)}, volume
  439 of {\em NATO Adv. Sci. Inst. Ser. C: Math. Phys. Sci.}, pages 131--186.
  Kluwer Acad. Publ., Dordrecht, 1994.
\newblock Notes partially written by Estela A. Gavosto.

\bibitem{friedland-milnor}
Shmuel Friedland and John Milnor.
\newblock Dynamical properties of plane polynomial automorphisms.
\newblock {\em Ergodic Theory Dynam. Systems}, 9(1):67--99, 1989.

\bibitem{Lamy:Tits}
St\'{e}phane Lamy.
\newblock L'alternative de {T}its pour {${\rm Aut}[{\Bbb C}^2]$}.
\newblock {\em J. Algebra}, 239(2):413--437, 2001.

\bibitem{Sibony:Panorama}
Nessim Sibony.
\newblock Dynamique des applications rationnelles de {$\bold P^k$}.
\newblock In {\em Dynamique et g\'{e}om\'{e}trie complexes ({L}yon, 1997)},
  volume~8 of {\em Panor. Synth\`eses}, pages ix--x, xi--xii, 97--185. Soc.
  Math. France, Paris, 1999.

\bibitem{smillie:entropy_henon}
John Smillie.
\newblock The entropy of polynomial diffeomorphisms of {${\bf C}^2$}.
\newblock {\em Ergodic Theory Dynam. Systems}, 10(4):823--827, 1990.

\end{thebibliography}
\end{document}